\documentclass[11pt, final]{article}
\usepackage{a4}
\usepackage{amsmath}%
\usepackage{amstext}%
\usepackage{amssymb}%
\usepackage{showkeys}%
\usepackage{epsfig}%
\usepackage{cite}
\usepackage{tikz}
\usepackage{subfig}
\usepackage{caption}
\usepackage{float}
\usepackage{multirow}
\usepackage{booktabs}

\usepackage{listings}
\newtheorem{theorem}{Theorem}

\newtheorem{algorithm}[theorem]{Algorithm}

\newtheorem{lemma}[theorem]{Lemma}

\newtheorem{pb}[theorem]{Problem}
\newenvironment{problem}{\pb\rm}{\endpb}

\newtheorem{rem}[theorem]{Remark}
\newenvironment{remark}{\rem\rm}{\endrem}

\newcounter{unnumber}

\newtheorem{prf}[unnumber]{Proof.}
\newenvironment{proof}{\prf\rm}{\hfill{$\blacksquare$}\endprf}
\newcommand{\R}{\mathbb{R}}%
\newcommand{\N}{\mathbb{N}}%
\newcommand{\ol}{\overline}%

\DeclareMathOperator*\dom{dom}%
\DeclareMathOperator*\B{\overline{\R}}%
\DeclareMathOperator*\gr{Gr}%
\DeclareMathOperator*\ran{ran}%
\DeclareMathOperator*\id{Id}%
\DeclareMathOperator*\argmin{argmin}

\DeclareMathOperator*\zer{zer}

\title{Forward-Backward and Tseng's Type Penalty Schemes for Monotone Inclusion Problems}

\author{Radu Ioan Bo\c{t} \thanks{Department of Mathematics, Chemnitz University of Technology,
D-09107 Chemnitz, Germany, e-mail: radu.bot@mathematik.tu-chemnitz.de. Research partially supported by DFG (German Research Foundation), project BO 2516/4-1.} \and Ern\"{o} Robert Csetnek \thanks {Department of Mathematics, Chemnitz University of Technology, D-09107 Chemnitz, Germany, e-mail: robert.csetnek@mathematik.tu-chemnitz.de. Research supported by DFG (German Research Foundation), project BO 2516/4-1.} }

\begin{document}
\maketitle

\noindent \textbf{Abstract.} We deal with monotone inclusion problems of the form $0\in Ax+Dx+N_C(x)$ in real Hilbert spaces, where $A$ is a maximally monotone operator, $D$ a cocoercive operator and $C$ the nonempty set of zeros of another cocoercive operator. We propose a forward-backward penalty algorithm for solving this problem which extends the one proposed by H. Attouch, M.-O. Czarnecki and J. Peypouquet in \cite{att-cza-peyp-c}. The condition which guarantees the weak ergodic convergence of the sequence of iterates generated by the proposed scheme is formulated by means of the Fitzpatrick function associated to the maximally monotone operator that describes the set $C$. In the second part we introduce a forward-backward-forward algorithm for monotone inclusion problems having the same structure, but this time by replacing the cocoercivity hypotheses with Lipschitz continuity conditions. The latter penalty type algorithm opens the gate to handle monotone inclusion problems with more complicated structures, for instance, involving compositions of maximally monotone operators with linear continuous ones. \vspace{1ex}

\noindent \textbf{Key Words.} maximally monotone operator, Fitzpatrick function, resolvent, cocoercive operator, Lipschitz continuous operator, forward-backward algorithm, forward-backward-forward algorithm, subdifferential, Fenchel conjugate\vspace{1ex}

\noindent \textbf{AMS subject classification.} 47H05, 65K05, 90C25

\section{Introduction and preliminaries}\label{sec-intr}

\subsection{Motivation and problem formulation}

In the last years one can observe in the literature an increasing interest in solving variational inequalities expressed as monotone inclusion problems of the form
\begin{equation}\label{att-cza-peyp-p}
0\in Ax+N_C(x),
\end{equation}
where ${\cal H}$ is a real Hilbert space, $A:{\cal H}\rightrightarrows{\cal H}$ is a maximally monotone operator, $C=\argmin \Psi$ is the set of global minima of the  proper, convex and lower semicontinuous function $\Psi : {\cal H} \rightarrow \overline{\R}:=\R\cup\{\pm\infty\}$ fulfilling $\min \Psi=0$ and $N_C:{\cal H}\rightrightarrows{\cal H}$ is the normal cone of the set $C \subseteq {\cal H}$ (see \cite{att-cza-10, att-cza-peyp-c, att-cza-peyp-p, noun-peyp, peyp-12}). Specifically, on can find in the literature forward-backward algorithms for solving \eqref{att-cza-peyp-p} (see \cite{att-cza-peyp-c, att-cza-peyp-p, noun-peyp, peyp-12}), which perform in each iteration a proximal step with respect to $A$ and a subgradient step with respect to the penalization function $\Psi$.

In case $\Psi : {\cal H} \rightarrow \R$ is differentiable with Lipschitz continuous gradient, for the algorithm, that reads as
$$\verb"Choose" \ x_1 \in {\cal H}. \ \verb"For" \ n \in \N \ \verb"set" \ x_{n+1}=(\id + \lambda_n A )^{-1}(x_n-\lambda_n \beta_n \nabla\Psi(x_n)),$$
with $(\lambda_n)_{n \in \N}$ and $(\beta_n)_{n \in \N}$ sequences of positive real numbers, ergodic convergence results are usually obtained in the following hypotheses
$$(H)\left\{
\begin{array}{lll}
(i) \ A+N_C \mbox{ is maximally monotone and }\{x \in {\cal H} : 0\in Ax+N_C(x)\}\neq\emptyset;\\
(ii) \ \mbox{ For every }p\in\ran N_C, \sum_{n \in \N} \lambda_n\beta_n\left[\Psi^*\left(\frac{p}{\beta_n}\right)-\sigma_C\left(\frac{p}{\beta_n}\right)\right]<+\infty;\\
(iii) \ (\lambda_n)_{n \in \N}\in\ell^2\setminus\ell^1.\end{array}\right.$$
Here, $\Psi^* : {\cal H} \rightarrow \overline \R$ denotes the Fenchel conjugate function of $\Psi$ and $\ran N_C$ the range of the normal cone operator $N_C:{\cal H}\rightrightarrows{\cal H}$. Let us mention that hypothesis (ii), which is the discretized counterpart of a condition introduced in \cite{att-cza-10} in the context of  continuous-time nonautonomous differential inclusions, is satisfied, if $\sum_{n \in \N} \frac{\lambda_n}{\beta_n}<+\infty$ and $\Psi$ is bounded below by a multiple of the square of the distance to $C$ (see \cite{att-cza-peyp-p}). This is for instance the case when $C=\zer L = \{x \in {\cal H}: Lx=0\}$, $L : {\cal H} \rightarrow {\cal H}$ is a linear continuous operator with closed range and $\Psi : {\cal H} \rightarrow \R, \Psi(x)=\|Lx\|^2$ (see \cite{att-cza-peyp-p, att-cza-peyp-c}). For further situations for which condition (ii) is fulfilled we refer to \cite[Section 4.1]{att-cza-peyp-c}.

It is worth mentioning that when $A$ is the convex subdifferential of a proper, convex and lower semicontinuous function $\Phi : {\cal H} \rightarrow \overline \R$ the above algorithm provides an iterative scheme for solving convex optimization problems which can be formulated as
\begin{equation}\label{opt-pr}
\min_{x\in{\cal H}}\{\Phi(x):x\in\argmin\Psi\}.
\end{equation}

Motivated by these considerations, we deal in this paper with monotone inclusion problems of the form
\begin{equation}\label{att-cza-peyp-p-gen}
0\in Ax+Dx+N_C(x),
\end{equation}
where $A :{\cal H}\rightrightarrows{\cal H}$ is a maximally monotone operator, $D :{\cal H}\rightarrow {\cal H}$ is a (single-valued) cocoercive operator and $C \subseteq {\cal H}$ is the (nonempty) set of zeros of another cocoercive operator $B :{\cal H}\rightarrow {\cal H}$. Following \cite{att-cza-peyp-c} we propose a forward-backward penalty algorithm for solving \eqref{att-cza-peyp-p-gen} and prove weak ergodic convergence for the generated sequence of iterates under hypotheses which generalize the ones in $(H)$. To this end we specially generalize (ii) to a condition which involves the Fitzpatrick function associated to the maximally monotone operator $B$. Added to that, we prove strong convergence for the sequence of iterates whenever $A$ is strongly monotone.

The investigations made in this manuscript are completed in Section \ref{sec3} with the treatment of the monotone inclusion problem \eqref{att-cza-peyp-p-gen}, this time by relaxing the cocoercivity of $D$ and $B$ to monotonicity and Lipschitz continuity. We formulate in this more general setting a forward-backward-forward penalty type algorithm for solving \eqref{att-cza-peyp-p-gen} and study its convergence properties. The interest in having a suitable algorithmic scheme in this context is given by the fact that it allows via some primal-dual techniques to deal with monotone inclusion problems having more complicated structures, for instance, involving compositions of maximally monotone operators with linear continuous ones.

\subsection{Notations and preliminary results}

For the readers convenience we present first some notations which are used throughout the paper (see \cite{bo-van, b-hab, bauschke-book, EkTem, simons, Zal-carte}). By $\N = \{1,2,...\}$ we denote the set of \textit{positive integer numbers} and let ${\cal H}$ be a real Hilbert space with \textit{inner product} $\langle\cdot,\cdot\rangle$ and associated \textit{norm} $\|\cdot\|=\sqrt{\langle \cdot,\cdot\rangle}$. The symbols $\rightharpoonup$ and $\rightarrow$ denote weak and strong convergence, respectively. When ${\cal G}$ is another Hilbert space and $K:{\cal H} \rightarrow {\cal G}$ a linear continuous operator, then the \textit{norm} of $K$ is defined as $\|K\| = \sup\{\|Kx\|: x \in {\cal H}, \|x\| \leq 1\}$, while $K^* : {\cal G} \rightarrow {\cal H}$, defined by $\langle K^*y,x\rangle = \langle y,Kx \rangle$ for all $(x,y) \in {\cal H} \times {\cal G}$, denotes the \textit{adjoint operator} of $K$.

For a function $f:{\cal H}\rightarrow\overline{\R}$ we denote by $\dom f=\{x\in {\cal H}:f(x)<+\infty\}$ its \textit{effective domain} and say that $f$ is \textit{proper}, if $\dom f\neq\emptyset$ and $f(x)\neq-\infty$ for all $x\in {\cal H}$.  Let $f^*:{\cal H} \rightarrow \overline \R$, $f^*(u)=\sup_{x\in {\cal H}}\{\langle u,x\rangle-f(x)\}$ for all $u\in {\cal H}$, be the \textit{conjugate function} of $f$. The \textit{subdifferential} of $f$ at $x\in {\cal H}$, with $f(x)\in\R$, is the set $\partial f(x):=\{v\in {\cal H}:f(y)\geq f(x)+\langle v,y-x\rangle \ \forall y\in {\cal H}\}$. We take by convention $\partial f(x):=\emptyset$, if $f(x)\in\{\pm\infty\}$. We also denote by $\min f := \inf_{x \in {\cal H}} f(x)$ and by $\argmin f :=\{x \in {\cal H}: f(x) = \min f \}$.

Let $S\subseteq {\cal H}$ be a nonempty set. The \textit{indicator function} of $S$, $\delta_S:{\cal H}\rightarrow \overline{\R}$, is the function which takes the value $0$ on $S$ and $+\infty$ otherwise. The subdifferential of the indicator function is the \textit{normal cone} of $S$, that is $N_S(x)=\{u\in {\cal H}:\langle u,y-x\rangle\leq 0 \ \forall y\in S\}$, if $x\in S$ and $N_S(x)=\emptyset$ for $x\notin S$. Notice that for $x\in S$, $u\in N_S(x)$ if and only if $\sigma_S(u)=\langle u,x\rangle$, where $\sigma_S$ is the support function of $S$, defined by $\sigma_S(u)=\sup_{y\in S}\langle y,u\rangle$.

For an arbitrary set-valued operator $M:{\cal H}\rightrightarrows {\cal H}$ we denote by $\gr M=\{(x,u)\in {\cal H}\times {\cal H}:u\in Mx\}$ its \emph{graph}, by $\dom M=\{x \in {\cal H} : Mx \neq \emptyset\}$ its \emph{domain}, by $\ran M=\{u\in {\cal H}: \exists x\in {\cal H} \mbox{ s.t. }u\in Mx\}$ its \emph{range} and by $M^{-1}:{\cal H}\rightrightarrows {\cal H}$ its \emph{inverse operator}, defined by $(u,x)\in\gr M^{-1}$ if and only if $(x,u)\in\gr M$. We use also the notation $\zer M=\{x\in {\cal H}: 0\in Mx\}$ for the \textit{set of zeros} of the operator $M$. We say that $M$ is \emph{monotone} if $\langle x-y,u-v\rangle\geq 0$ for all $(x,u),(y,v)\in\gr M$. A monotone operator $M$ is said to be \emph{maximally monotone}, if there exists no proper monotone extension of the graph of $M$ on ${\cal H}\times {\cal H}$. Let us mention that in case $M$ is maximally monotone, $\zer M$ is a convex and closed set \cite[Proposition 23.39]{bauschke-book}. We refer to \cite[Section 23.4]{bauschke-book} for conditions ensuring that $\zer M$ is nonempty.  If $M$ is maximally monotone, then one has the following characterization for the set of its zeros
\begin{equation}\label{charact-zeros-max}
z\in\zer M \mbox{ if and only if }\langle w,u-z\rangle\geq 0\mbox{ for all }(u,w)\in \gr M.
\end{equation}

The operator $M$ is said to be \emph{$\gamma$-strongly monotone} with $\gamma>0$, if $\langle x-y,u-v\rangle\geq \gamma\|x-y\|^2$ for all $(x,u),(y,v)\in\gr M$. Notice that if $M$ is maximally monotone and strongly monotone, then $\zer M$ is a singleton, thus nonempty (see \cite[Corollary 23.37]{bauschke-book}).

The \emph{resolvent} of $M$, $J_M:{\cal H} \rightrightarrows {\cal H}$, is defined by $J_M=(\id+M)^{-1}$, where $\id :{\cal H} \rightarrow {\cal H}, \id(x) = x$ for all $x \in {\cal H}$, is the \textit{identity operator} on ${\cal H}$. Moreover, if $M$ is maximally monotone, then $J_M:{\cal H} \rightarrow {\cal H}$ is single-valued and maximally monotone (cf. \cite[Proposition 23.7 and Corollary 23.10]{bauschke-book}). For an arbitrary $\gamma>0$ we have (see \cite[Proposition 23.18]{bauschke-book})
\begin{equation}\label{j-inv-op}
J_{\gamma M}+\gamma J_{\gamma^{-1}M^{-1}}\circ \gamma^{-1}\id=\id.
\end{equation}

The \emph{Fitzpatrick function} associated to a monotone operator $M$, defined as
$$\varphi_M:{\cal H}\times {\cal H}\rightarrow \B, \ \varphi_M(x,u)=\sup_{(y,v)\in\gr M}\{\langle x,v\rangle+\langle y,u\rangle-\langle y,v\rangle\},$$
is a convex and lower semicontinuous function and it will play an important role throughout the paper. Introduced by Fitzpatrick in \cite{fitz}, this notion opened the gate towards the employment of convex analysis specific tools when investigating the maximality of monotone operators (see \cite{bauschke-book, bausch-m-s, b-hab, BCW-set-val, bo-van, borw-06, bu-sv-02, simons} and the references therein). In case $M$ is maximally monotone, $\varphi_M$ is proper and it fulfils
$$\varphi_M(x,u)\geq \langle x,u\rangle \ \forall (x,u)\in {\cal H}\times {\cal H},$$
with equality if and only if $(x,u)\in\gr M$. Notice that if $f:{\cal H}\rightarrow\B$, is a proper, convex and lower semi-continuous function, then $\partial f$ is a maximally monotone operator (cf. \cite{rock}) and it holds $(\partial f)^{-1} = \partial f^*$. Furthermore, the following inequality is true (see \cite{bausch-m-s})
\begin{equation}\label{fitzp-subdiff-ineq}
 \varphi_{\partial f}(x,u)\leq f(x) +f^*(u) \ \forall (x,u)\in {\cal H}\times {\cal H}.
\end{equation}
We refer the reader to \cite{bausch-m-s}, for formulae of the corresponding Fitzpatrick functions computed for particular classes of monotone operators.

Let $\gamma>0$ be arbitrary. A single-valued operator $M:{\cal H}\rightarrow {\cal H}$ is said to be \textit{$\gamma$-cocoercive}, if $\langle x-y,Mx-My\rangle\geq \gamma\|Mx-My\|^2$ for all $(x,y)\in {\cal H}\times {\cal H}$, and \textit{$\gamma$-Lipschitz continuous}, if $\|Mx-My\|\leq \gamma\|x-y\|$ for all $(x,y)\in {\cal H}\times {\cal H}$. A single-valued linear operator $M:{\cal H} \rightarrow {\cal H}$ is said to be \textit{skew}, if $\langle x,Mx\rangle =0$ for all $x \in {\cal H}$.

We close the section by presenting some convergence results that will be used several times in the paper.  Let $(x_n)_{n \in \N}$ be a sequence in ${\cal H}$ and $(\lambda_k)_{k \in \N}$ a sequence of positive numbers such that $\sum_{k \in \N} \lambda_k=+\infty$. Let $(z_n)_{n \in \N} $ be the sequence of weighted averages defined as (see \cite{att-cza-peyp-c})
\begin{equation}\label{average}
z_n=\frac{1}{\tau_n}\sum_{k=1}^n\lambda_k x_k, \mbox{ where }\tau_n= \sum_{k=1}^n\lambda_k \ \forall n \in \N.
\end{equation}

\begin{lemma}\label{opial-passty} (Opial-Passty) Let $F$ be a nonempty subset of ${\cal H}$ and assume that the limes $\lim_{n\rightarrow\infty}\|x_n-x\|$ exists for every $x\in F$. If every weak cluster point of $(x_n)_{n \in \N}$ (respectively $(z_n)_{n \in \N}$) lies in $F$, then $(x_n)_{n \in \N}$ (respectively $(z_n)_{n \in \N}$) converges weakly to an element in $F$ as $n\rightarrow+\infty$.
\end{lemma}

The following result is taken from \cite{att-cza-peyp-c}.

\begin{lemma}\label{Fejer-real-seq} Let $(a_n)_{n \in \N}$, $(b_n)_{n \in \N}$ and $(\varepsilon_n)_{n \in \N}$ be real sequences. Assume that $(a_n)_{n \in \N}$ is bounded from below, $(b_n)_{n \in \N}$ is nonnegative, $(\varepsilon_n)_{n \in \N} \in\ell^1$ and $a_{n+1}-a_n+b_n\leq\varepsilon_n$ for any $n\in\N$. Then $(a_n)_{n \in \N}$ is convergent and $(b_n)_{n \in \N} \in\ell^1$.
\end{lemma}

\section{Forward-Backward Penalty Schemes}\label{sec2}

The problem we deal with at the beginning of this section has the following formulation.

\begin{problem}\label{pr-set-val}
Let ${\cal H}$ be a real Hilbert space, $A,B:{\cal H}\rightrightarrows {\cal H}$ maximally monotone operators, $D:{\cal H}\rightarrow{\cal H}$ an $\eta$-cocoercive operator with $\eta>0$ and suppose that $C=\zer B\neq\emptyset$. The monotone inclusion problem to solve is
$$0 \in Ax+Dx+N_C(x).$$
\end{problem}

The following iterative scheme for solving Problem \ref{pr-set-val} is inspired by \cite{att-cza-peyp-c}.

\begin{algorithm}\label{alg-fb-set-val}$ $

\noindent\begin{tabular}{rl}
\verb"Initialization": & \verb"Choose" $x_1\in{\cal H}$\\
\verb"For" $n \in \N$: & \verb"Choose" $w_n\in Bx_n$\\
                     & \verb"Set" $x_{n+1}=J_{\lambda_n A}(x_n-\lambda_n D x_n-\lambda_n\beta_n w_n)$,
\end{tabular}
\end{algorithm}
where $(\lambda_n)_{n \in \N} $ and $(\beta_n)_{n \in \N}$ are sequences of positive real numbers. Notice that Algorithm  \ref{alg-fb-set-val} is well-defined, if $\dom B={\cal H}$, which will be the case in Subsection \ref{subsec2.2}, when $B$ is assumed to be cocoercive. For the convergence statement the following hypotheses are needed
$$(H_{fitz})\left\{
\begin{array}{lll}
(i) \ A+N_C \mbox{ is maximally monotone and }\zer(A+D+N_C)\neq\emptyset;\\
(ii) \ \mbox{ For every }p\in\ran N_C, \sum_{n \in \N} \lambda_n\beta_n\left[\sup\limits_{u\in C}\varphi_B\left(u,\frac{p}{\beta_n}\right)-\sigma_C\left(\frac{p}{\beta_n}\right)\right]<+\infty;\\
(iii) \ (\lambda_n)_{n \in \N} \in\ell^2\setminus\ell^1.\end{array}\right.$$
Since $A$ is maximally monotone and $C$ is a nonempty convex and closed set, $A + N_C$ is maximally monotone if a so-called regularity condition is fulfilled. We refer the reader to
\cite{bauschke-book, b-hab, BCW-set-val, bo-van, borw-06, simons, Zal-carte} for conditions guaranteeing the maximal monotonicity of the sum of two maximally monotone operators.

Further, as $D$ is maximally monotone (see \cite[Example 20.28]{bauschke-book}) and $\dom D={\cal H}$, the hypothesis (i) above guarantees that $A+D+N_C$ is maximally monotone, too (see \cite[Corollary 24.4]{bauschke-book}). Moreover, for each $p\in\ran N_C$ we have
$$\sup\limits_{u\in C}\varphi_B\left(u,\frac{p}{\beta_n}\right)-\sigma_C\left(\frac{p}{\beta_n}\right)\geq 0 \ \forall n\in\N.$$
Indeed, if $p\in\ran N_C$, then there exists $\ol u\in C$ such that $p\in N_C(\ol u)$. This implies that
$$\sup\limits_{u\in C}\varphi_B\left(u,\frac{p}{\beta_n}\right)-\sigma_C\left(\frac{p}{\beta_n}\right)\geq \left\langle \ol u,\frac{p}{\beta_n}\right\rangle-\sigma_C\left(\frac{p}{\beta_n}\right)=0 \ \forall n \in \N.$$

\begin{remark}\label{fitz-conj} Let us mention that, if $Dx=0$ for all $x\in{\cal H}$ and $B=\partial\Psi$, where $\Psi:{\cal H}\rightarrow\B$ is a proper, convex and lower semicontinuous function with $\min \Psi=0$, then the monotone inclusion problem in  Problem \ref{pr-set-val} becomes \eqref{att-cza-peyp-p}, since in this case $C=\argmin\Psi$. Moreover, as $\Psi(x)=0$ for all $x\in C$, by \eqref{fitzp-subdiff-ineq} it follows that condition (ii) in $(H)$ implies condition (ii) in $(H_{fitz})$, hence the hypothesis formulated by means of the Fitzpatrick function extends the one given \cite{att-cza-peyp-c} to the more general setting considered in Problem \ref{pr-set-val}.
\end{remark}

The techniques used as follows extend the ones from \cite{att-cza-peyp-c} to the general context of the monotone inclusion problem stated in Problem \ref{pr-set-val}.

\subsection{The general case}\label{subsec2.1}

In this subsection we will prove an abstract convergence result for Algorithm \eqref{alg-fb-set-val}, which will be subsequently refined in the case when $B$ is a cocoercive operator.

\begin{lemma}\label{fb-ineq1-set-v} Let $(x_n)_{n \in \N}$ and $(w_n)_{n \in \N}$ be the sequences generated by Algorithm \ref{alg-fb-set-val} and take $(u,w)\in\gr (A+D+N_C)$ such that $w=v+p+Du$, where $v\in Au$ and $p\in N_C(u)$. Then the following inequality holds for any $n\in\N$
\begin{align}\label{lem6}
& \|x_{n+1}-u\|^2-\|x_n-u\|^2+\lambda_n(2\eta-3\lambda_n)\|Dx_n-Du\|^2\leq \nonumber \\
& 2\lambda_n\beta_n\left[\sup_{u\in C}\varphi_B\left(u,\frac{p}{\beta_n}\right)-\sigma_C\left(\frac{p}{\beta_n}\right)\right]+3\lambda_n^2\beta_n^2\|w_n\|^2+3\lambda_n^2\|Du+v\|^2+2\lambda_n\langle w,u-x_n\rangle.
\end{align}
\end{lemma}

\begin{proof}
From the definition of the resolvent of $A$ we have $\frac{x_n-x_{n+1}}{\lambda_n}-\beta_nw_n-Dx_n\in Ax_{n+1}$ and since $v\in Au$, the monotonicity of $A$ guarantees
\begin{equation}\label{A-mon-fb}\langle x_n-x_{n+1}-\lambda_n(\beta_nw_n+Dx_n+v),x_{n+1}-u\rangle\geq 0 \ \forall n \in \N,
\end{equation}
thus
$$\langle x_n-x_{n+1},u-x_{n+1}\rangle\leq \lambda_n\langle \beta_nw_n+Dx_n+v,u-x_{n+1}\rangle \ \forall n \in \N.$$

Further, since $$\langle x_n-x_{n+1},u-x_{n+1}\rangle=\frac{1}{2}\|x_{n+1}-u\|^2-\frac{1}{2}\|x_n-u\|^2+\frac{1}{2}\|x_{n+1}-x_n\|^2,$$ we get for any $n \in \N$
\begin{align*}
& \ \|x_{n+1}-u\|^2-\|x_n-u\|^2\leq 2\lambda_n\langle \beta_nw_n+Dx_n+v,u-x_{n+1}\rangle-\|x_{n+1}-x_n\|^2\\
= & \ 2\lambda_n\langle \beta_nw_n+Dx_n+v,u-x_n\rangle+2\lambda_n\langle \beta_nw_n+Dx_n+v,x_n-x_{n+1}\rangle-\|x_{n+1}-x_n\|^2\\
\leq & \ 2\lambda_n\langle \beta_nw_n+Dx_n+v,u-x_n\rangle+\lambda_n^2\|\beta_nw_n+Dx_n+v\|^2\\
\leq & \ 2\lambda_n\langle \beta_nw_n+Dx_n+v,u-x_n\rangle+3\lambda_n^2\beta_n^2\|w_n\|^2+3\lambda_n^2\|Du+v\|^2+3\lambda_n^2\|Dx_n-Du\|^2.
\end{align*}
Next we evaluate the first term on the right hand-side of the last of the above inequalities. By using the cocoercivity of $D$ and the definition of the Fitzpatrick function and that $w_n\in Bx_n$ and $\sigma_C\left(\frac{p}{\beta_n}\right)=\langle u,\frac{p}{\beta_n}\rangle$ for every $n \in \N$, we obtain
\begin{align*}
& \ 2\lambda_n\langle \beta_nw_n+Dx_n+v,u-x_n\rangle\\
= & \ 2\lambda_n\langle \beta_nw_n+Dx_n+w-p-Du,u-x_n\rangle\\
= & \ 2\lambda_n\langle Dx_n-Du,u-x_n\rangle+2\lambda_n\langle \beta_nw_n-p,u-x_n\rangle+2\lambda_n\langle w,u-x_n\rangle\\
= & \ 2\lambda_n\langle Dx_n-Du,u-x_n\rangle+2\lambda_n\beta_n\left(\langle w_n,u\rangle+\left\langle x_n,\frac{p}{\beta_n}\right\rangle-\langle w_n,x_n\rangle
-\left\langle \frac{p}{\beta_n},u\right\rangle\right)\\
& \ +2\lambda_n\langle w,u-x_n\rangle\\
\leq & \ -2\eta\lambda_n\|Dx_n-Du\|^2+2\lambda_n\beta_n\left[\sup_{u\in C}\varphi_B\left(u,\frac{p}{\beta_n}\right)-\sigma_C\left(\frac{p}{\beta_n}\right)\right]+2\lambda_n\langle w,u-x_n\rangle.
\end{align*}
This provides the desired conclusion.
\end{proof}

\begin{theorem}\label{fb-conv-set-v} Let $(x_n)_{n \in \N}$ and $(w_n)_{n \in \N}$ be the sequences generated by Algorithm \ref{alg-fb-set-val} and $(z_n)_{n \in \N}$ the sequence defined in \eqref{average}.  If $(H_{fitz})$ is fulfilled and $(\lambda_n\beta_n\|w_n\|)_{n \in \N} \in \ell^2$, then  $(z_n)_{n \in \N}$ converges weakly to an element in $\zer(A+D+N_C)$ as $n\rightarrow+\infty$.
\end{theorem}

\begin{proof} As $\lim_{n\rightarrow+\infty}\lambda_n=0$, there exists $n_0\in\N$ such that $2\eta-3\lambda_n\geq 0$ for all $n\geq n_0$. Thus, for $(u,w)\in\gr (A+D+N_C)$, such that $w=v+p+Du$, where $v\in Au$ and $p\in N_C(u)$,
by \eqref{lem6} it holds for any $n \geq n_0$
\begin{align}\label{help}
& \|x_{n+1}-u\|^2-\|x_n-u\|^2 \leq \nonumber\\
& 2\lambda_n\beta_n\left[\sup_{u\in C}\varphi_B\left(u,\frac{p}{\beta_n}\right)-\sigma_C\left(\frac{p}{\beta_n}\right)\right] + 3\lambda_n^2\beta_n^2\|w_n\|^2+3\lambda_n^2\|Du+v\|^2 + 2\lambda_n \langle w,u-x_n\rangle.
\end{align}
By Lemma \ref{opial-passty}, it is sufficient to prove that the following two statements hold:
\begin{itemize}
\item[(a)] for every $u\in\zer(A+D+N_C)$ the sequence $(\|x_n-u\|)_{n \in \N}$ is convergent;

\item [(b)] every weak cluster point of $(z_n)_{n \in \N}$ lies in $\zer(A+D+N_C)$.
\end{itemize}

(a) For every $u\in\zer(A+D+N_C)$ one can take $w=0$ in \eqref{help} and the conclusion follows from Lemma \ref{Fejer-real-seq}.

(b) Let $z$ be a weak cluster point of $(z_n)_{n \in \N}$. As we already noticed that $A+D+N_C$ is maximally monotone, in order to show that $z\in\zer(A+D+N_C)$ we will use the characterization given in \eqref{charact-zeros-max}.
Take $(u,w)\in\gr (A+D+N_C)$ such that $w=v+p+Du$, where $v\in Au$ and $p\in N_C(u)$. Let be $N\in\N$ with $N\geq n_0+2$. Summing up for $n=n_0+1,...,N$ the inequalities in \eqref{help}, we get
$$\|x_{N+1}-u\|^2-\|x_{n_0+1}-u\|^2\leq L+2\left\langle w, \sum_{n=1}^N\lambda_nu-\sum_{n=1}^N\lambda_nx_n-\sum_{n=1}^{n_0}\lambda_nu+\sum_{n=1}^{n_0}\lambda_nx_n\right\rangle,$$
where
\begin{align*}
L = & \ 2\sum_{n \geq n_0+1} \lambda_n\beta_n\left[\sup_{u\in C}\varphi_B\left(u,\frac{p}{\beta_n}\right)-\sigma_C\left(\frac{p}{\beta_n}\right)\right]\\
+ & \ 3\sum_{n \geq n_0+1} \lambda_n^2\beta_n^2\|w_n\|^2+3\sum_{n \geq n_0+1} \lambda_n^2\|Du+v\|^2 \in \R.
\end{align*}

Discarding the nonnegative term $\|x_{N+1}-u\|^2$ and dividing by $2\tau_N=2\sum_{k=1}^N\lambda_k$ we obtain
$$-\frac{\|x_{n_0+1}-u\|^2}{2\tau_N}\leq \frac{\widetilde{L}}{2\tau_N}+\langle w,u-z_N\rangle,$$
where $\widetilde{L}:=L+2\langle w, -\sum_{n=1}^{n_0}\lambda_nu+\sum_{n=1}^{n_0}\lambda_nx_n\rangle \in \R$. By passing to the limit as $N\rightarrow+\infty$ and using that $\lim_{N\rightarrow+\infty}\tau_N=+\infty$, we get $$\liminf_{N\rightarrow+\infty}\langle w,u-z_N\rangle\geq 0.$$
Since $z$ is a weak cluster point of $(z_n)_{n \in \N}$, we obtain that $\langle w,u-z\rangle\geq 0$. Finally, as this inequality holds for arbitrary $(u,w)\in\gr (A+D+N_C)$, the desired conclusion follows.
\end{proof}

In the following we show that strong monotonicity of the operator $A$ ensures strong convergence of the sequence $(x_n)_{n \in \N}$.

\begin{theorem}\label{str-mon-fb} Let $(x_n)_{n \in \N}$ and $(w_n)_{n \in \N}$ be the sequences generated by Algorithm \ref{alg-fb-set-val}. If $(H_{fitz})$ is fulfilled, $(\lambda_n\beta_n\|w_n\|)_{n \in \N} \in \ell^2$ and the operator $A$ is $\gamma$-strongly monotone with $\gamma>0$, then $(x_n)_{n \in \N}$ converges strongly to the unique element in $\zer(A+D+N_C)$ as $n\rightarrow+\infty$.
\end{theorem}

\begin{proof} Let be $u \in \zer(A+D+N_C)$ and $w=0=v+p+Du$, where $v \in Au$ and $p \in N_C(u)$. Since $A$ is $\gamma$-strongly monotone, inequality \eqref{A-mon-fb} becomes
\begin{equation}\label{A-str-mon-fb}
\langle x_n-x_{n+1}-\lambda_n(\beta_nw_n+Dx_n+v),x_{n+1}-u\rangle\geq \lambda_n\gamma\|x_{n+1}-u\|^2 \ \forall n \in \N.
\end{equation}
Following the lines of the proof of Lemma \ref{fb-ineq1-set-v} for $w=0$ we obtain for any $n \in \N$
\begin{align*}
& 2\gamma\lambda_n\|x_{n+1}-u\|^2+\|x_{n+1}-u\|^2-\|x_n-u\|^2+\lambda_n(2\eta-3\lambda_n)\|Dx_n-Du\|^2\leq\\
& 2\lambda_n\beta_n\left[\sup_{u\in C}\varphi_B\left(u,\frac{p}{\beta_n}\right)-\sigma_C\left(\frac{p}{\beta_n}\right)\right]+3\lambda_n^2\beta_n^2\|w_n\|^2+3\lambda_n^2\|Du+v\|^2.
\end{align*}
Thus, as $\lim_{n\rightarrow+\infty}\lambda_n=0$, there exists $n_0\in\N$ such that for all $n\geq n_0$
\begin{align*}
& \ 2\gamma\lambda_n\|x_{n+1}-u\|^2+\|x_{n+1}-u\|^2-\|x_n-u\|^2\\
\leq & \ 2\lambda_n\beta_n\left[\sup_{u\in C}\varphi_B\left(u,\frac{p}{\beta_n}\right)-\sigma_C\left(\frac{p}{\beta_n}\right)\right]+3\lambda_n^2\beta_n^2\|w_n\|^2+3\lambda_n^2\|Du+v\|^2
\end{align*}
and, so,
\begin{eqnarray*}
2\gamma\sum_{n \geq n_0} \lambda_n\|x_{n+1}-u\|^2 & \leq & \|x_{n_0}-u\|^2+2\sum_{n \geq n_0} \lambda_n\beta_n\left[\sup_{u\in C}\varphi_B\left(u,\frac{p}{\beta_n}\right)-\sigma_C\left(\frac{p}{\beta_n}\right)\right]\\
&  & +3\sum_{n \geq n_0} \lambda_n^2\beta_n^2\|w_n\|^2 +3\|Du+v\|^2\sum_{n \geq n_0} \lambda_n^2<+\infty.
\end{eqnarray*}
Since $\sum_{n \in \N} \lambda_n=+\infty$ and  $(\|x_n-u\|)_{n \in \N}$ is convergent (see the proof of Theorem \ref{fb-conv-set-v} (a)), it follows $\lim_{n\rightarrow+\infty}\|x_n-u\|=0$.
\end{proof}

\subsection{The case $B$ is cocoercive}\label{subsec2.2}

In this subsection we deal with the situation when $B$ is a (single-valued) cocoercive operator. Our aim is to show that in this setting the assumption $(\lambda_n\beta_n\|w_n\|)_{n \in \N} \in\ell^2$ in Theorem \ref{fb-conv-set-v} and Theorem \ref{str-mon-fb} can be replaced by a milder condition involving only the sequences $(\lambda_n)_{n \in \N} $ and $(\beta_n)_{n \in \N}$. The problem under consideration has the following formulation.

\begin{problem}\label{pr-cocoercive-single-val}
Let ${\cal H}$ be a real Hilbert space, $A:{\cal H}\rightrightarrows {\cal H}$  a maximally monotone operator, $D:{\cal H}\rightarrow{\cal H}$ an $\eta$-cocoercive operator with $\eta>0$, $B:{\cal H}\rightarrow{\cal H}$ a $\mu$-cocoercive operator with $\mu>0$ and suppose that $C=\zer B \neq\emptyset$. The monotone inclusion problem to solve is
$$0\in Ax+Dx+N_C(x).$$
\end{problem}

Algorithm \ref{alg-fb-set-val} has in this particular setting the following formulation.

\begin{algorithm}\label{alg-fb-single-val}$ $

\noindent\begin{tabular}{rl}
\verb"Initialization": & \verb"Choose" $x_1\in{\cal H}$\\
\verb"For" $n \in \N$ \verb"set": & $x_{n+1}=J_{\lambda_n A}(x_n-\lambda_n D x_n-\lambda_n\beta_n Bx_n)$.
\end{tabular}
\end{algorithm}

\begin{remark} (a) If $Dx=0$ for every $x\in{\cal H}$ and $B=\partial\Psi$, where $\Psi:{\cal H}\rightarrow \R$ is a convex and differentiable function with $\mu^{-1}$-Lipschitz gradient for $\mu > 0$ fulfilling $\min \Psi=0$, then we rediscover the setting considered in \cite[Section 3]{att-cza-peyp-c}, while Algorithm \ref{alg-fb-single-val} becomes the iterative method investigated in that paper.

(b) In case $Bx=0$ for all $x\in{\cal H}$ Algorithm \ref{alg-fb-single-val} turns out to be classical forward-backward scheme (see \cite{bauschke-book, combettes, vu}), since in this case $C={\cal H}$, hence $N_C(x)=\{0\}$ for all $x\in{\cal H}$.
\end{remark}

Before stating the convergence result for Algorithm \ref{alg-fb-single-val} some technical results are in order.

\begin{lemma}\label{ineq1-fb-single-val} Let be $u\in C\cap\dom A$ and $v\in Au$. Then for every $\varepsilon\geq 0$ and any $n\in\N$ we have
\begin{align}\label{lem12}
& \|x_{n+1}-u\|^2-\|x_n-u\|^2+\frac{\varepsilon}{1+\varepsilon}\|x_{n+1}-x_n\|^2+\frac{2\varepsilon}{1+\varepsilon}\lambda_n\beta_n\langle Bx_n,x_n-u\rangle \leq \nonumber \\
& \lambda_n\beta_n\left((1+\varepsilon)\lambda_n\beta_n-\frac{2\mu}{1+\varepsilon}\right)\|Bx_n\|^2+2\lambda_n\langle Dx_n+v,u-x_{n+1}\rangle.
\end{align}
\end{lemma}

\begin{proof} As in the proof of Lemma \ref{fb-ineq1-set-v} we obtain for any $n \in \N$ that
\begin{align*}
& \|x_{n+1}-u\|^2-\|x_n-u\|^2+\|x_{n+1}-x_n\|^2\leq 2\lambda_n\langle \beta_nBx_n+Dx_n+v,u-x_{n+1}\rangle =\\
& 2\lambda_n\beta_n\langle Bx_n,u-x_n\rangle+2\lambda_n\beta_n\langle Bx_n,x_n-x_{n+1}\rangle+2\lambda_n\langle Dx_n+v,u-x_{n+1}\rangle.
\end{align*}
Since $B$ is $\mu$-cocoercive and $Bu=0$ we have that
$$\langle Bx_n,u-x_n\rangle\leq -\mu\|Bx_n\|^2 \ \forall n \in \N,$$
hence
$$2\lambda_n\beta_n\langle Bx_n,u-x_n\rangle\leq -\frac{2\mu}{1+\varepsilon}\lambda_n\beta_n\|Bx_n\|^2+\frac{2\varepsilon}{1+\varepsilon}\lambda_n\beta_n\langle Bx_n,u-x_n\rangle \ \forall n \in \N \ \forall \varepsilon \geq 0.$$
Inequality \eqref{lem12} follows by taking into consideration also that
$$2\lambda_n\beta_n\langle Bx_n,x_n-x_{n+1}\rangle\leq \frac{1}{1+\varepsilon}\|x_{n+1}-x_n\|^2+(1+\varepsilon)\lambda_n^2\beta_n^2\|Bx_n\|^2 \ \forall n \in \N \ \forall \varepsilon \geq 0.$$
\end{proof}

\begin{lemma}\label{ineq2-fb-single-val} Assume that $\limsup_{n\rightarrow+\infty}\lambda_n\beta_n<2\mu$ and let be $u\in C\cap\dom A$ and $v\in Au$. Then there exist $a,b>0$ and $n_0 \in \N$ such that for any $n\geq n_0$ it holds
\begin{align}\label{lem13}
& \|x_{n+1}-u\|^2-\|x_n-u\|^2+a\left(\|x_{n+1}-x_n\|^2+\lambda_n\beta_n\langle Bx_n,x_n-u\rangle + \lambda_n\beta_n\|Bx_n\|^2\right) \leq \nonumber \\
& \left(b\lambda_n^2-2\eta\lambda_n\right)\|Dx_n-Du\|^2+2\lambda_n\langle v+Du,u-x_n\rangle+b\lambda_n^2\|Du+v\|^2.
\end{align}
\end{lemma}

\begin{proof} We start by noticing that, by making use of the cocoercivity of $D$, for every $\varepsilon > 0$ and any $n\in\N$ it holds
\begin{align*}
& \ 2\lambda_n\langle Dx_n+v,u-x_{n+1}\rangle\\
=& \ 2\lambda_n\langle Dx_n+v,x_n-x_{n+1}\rangle+2\lambda_n\langle Dx_n+v,u-x_n\rangle\\
\leq & \ \frac{\varepsilon}{2(1+\varepsilon)}\|x_{n+1}-x_n\|^2+\frac{2(1+\varepsilon)}{\varepsilon}\lambda_n^2\|Dx_n+v\|^2+2\lambda_n\langle Dx_n+v,u-x_n\rangle \\
\leq & \ \frac{\varepsilon}{2(1+\varepsilon)}\|x_{n+1}-x_n\|^2+\frac{4(1+\varepsilon)}{\varepsilon}\lambda_n^2\|Dx_n-Du\|^2+\frac{4(1+\varepsilon)}{\varepsilon}\lambda_n^2\|Du+v\|^2+\\
& \ 2\lambda_n\langle Dx_n-Du,u-x_n\rangle+2\lambda_n\langle v+Du,u-x_n\rangle\\
\leq & \ \frac{\varepsilon}{2(1+\varepsilon)}\|x_{n+1}-x_n\|^2+\frac{4(1+\varepsilon)}{\varepsilon}\lambda_n^2\|Dx_n-Du\|^2+\frac{4(1+\varepsilon)}{\varepsilon}\lambda_n^2\|Du+v\|^2-\\
& \ 2\lambda_n\eta\| Dx_n-Du\|^2+2\lambda_n\langle v+Du,u-x_n\rangle.
\end{align*}
In combination with \eqref{lem12} it yields for every $\varepsilon > 0$ and any $n\in\N$
\begin{align*}
& \ \|x_{n+1}-u\|^2-\|x_n-u\|^2+\frac{\varepsilon}{2(1+\varepsilon)}\|x_{n+1}-x_n\|^2+\frac{2\varepsilon}{1+\varepsilon}\lambda_n\beta_n\langle Bx_n,x_n-u\rangle\\
& \ + \frac{\varepsilon}{1+\varepsilon}\lambda_n\beta_n\|Bx_n\|^2\\
\leq & \ \lambda_n\beta_n\left((1+\varepsilon)\lambda_n\beta_n-\frac{2\mu}{1+\varepsilon}+\frac{\varepsilon}{1+\varepsilon}\right)\|Bx_n\|^2+ \left(\frac{4(1+\varepsilon)}{\varepsilon}\lambda_n^2-2\eta\lambda_n\right)\|Dx_n-Du\|^2\\
& \ +2\lambda_n\langle v+Du,u-x_n\rangle+\frac{4(1+\varepsilon)}{\varepsilon}\lambda_n^2\|Du+v\|^2.
\end{align*}
Since $\limsup_{n\rightarrow+\infty}\lambda_n\beta_n<2\mu$, there exists $\alpha > 0$ and $n_0 \in\N$ such that $\lambda_n\beta_n<\alpha < 2\mu$ for any $n\geq n_0$. Hence, for any $n \geq n_0$ and every $\varepsilon > 0$ it holds
$$\lambda_n\beta_n\left((1+\varepsilon)\lambda_n\beta_n-\frac{2\mu}{1+\varepsilon}+\frac{\varepsilon}{1+\varepsilon}\right) < \alpha \left((1+\varepsilon)\alpha-\frac{2\mu}{1+\varepsilon}+\frac{\varepsilon}{1+\varepsilon}\right)$$
and one can take $\varepsilon_0 > 0$ small enough such that $(1+\varepsilon_0)\alpha-\frac{2\mu}{1+\varepsilon_0}+\frac{\varepsilon_0}{1+\varepsilon_0} < 0$. Taking
$a=\frac{\varepsilon_0}{2(1+\varepsilon_0)}$ and $b=\frac{4(1+\varepsilon_0)}{\varepsilon_0}$ the desired conclusion follows.
\end{proof}

\begin{lemma}\label{ineq3-fb-single-val} Assume that $\limsup_{n\rightarrow+\infty}\lambda_n\beta_n<2\mu$ and $\lim_{n\rightarrow+\infty}\lambda_n=0$ and let be $(u,w)$ $\in\gr(A+D+N_C)$ such that $w=v+p+Du$, where $v\in Au$ and $p\in N_C(u)$. Then there exist $a,b >0$ and $n_1 \in \N$ such that for all $n \geq n_1$ it holds
\begin{align}\label{lem14}
& \ \|x_{n+1}-u\|^2-\|x_n-u\|^2+a\left(\|x_{n+1}-x_n\|^2+\frac{\lambda_n\beta_n}{2}\langle Bx_n,x_n-u\rangle + \lambda_n\beta_n\|Bx_n\|^2\right) \nonumber\\
\leq & \ \frac{a\lambda_n\beta_n}{2}\left[\sup_{u\in C}\varphi_B\left(u,\frac{4p}{a\beta_n}\right)-\sigma_C\left(\frac{4p}{a\beta_n}\right)\right]+2\lambda_n\langle w,u-x_n\rangle+b\lambda_n^2\|Du+v\|^2.
\end{align}
\end{lemma}

\begin{proof}
According to Lemma \ref{ineq2-fb-single-val}, there exist $a,b >0$ and $n_0 \in \N$ such that for any $n \geq n_0$ inequality \eqref{lem13} holds. Since $\lim_{n\rightarrow\infty}\lambda_n=0$, there exists $n_1\in\N, n_1 \geq n_0$ such that $b\lambda_n^2-2\eta\lambda_n\leq 0$ for all $n\geq n_1$, hence,
\begin{align*}
& \ \|x_{n+1}-u\|^2-\|x_n-u\|^2+a\left(\|x_{n+1}-x_n\|^2+\lambda_n\beta_n\langle Bx_n,x_n-u\rangle + \lambda_n\beta_n\|Bx_n\|^2\right) \\
\leq & \ 2\lambda_n\langle v+Du,u-x_n\rangle+b\lambda_n^2\|Du+v\|^2 \ \forall n \geq n_1.
\end{align*}
The conclusion follows by combining this inequality with the subsequent estimation that holds for any $n \in \N$:
\begin{align*}
& \ 2\lambda_n\langle v+Du,u-x_n\rangle+\frac{a\lambda_n\beta_n}{2}\langle Bx_n,u-x_n\rangle\\
= & \ 2\lambda_n\langle -p,u-x_n\rangle+\frac{a\lambda_n\beta_n}{2}\langle Bx_n,u-x_n\rangle+2\lambda_n\langle w,u-x_n\rangle\\
= & \ \frac{a\lambda_n\beta_n}{2}\left(\langle Bx_n,u\rangle +\left\langle x_n,\frac{4p}{a\beta_n}\right\rangle-\langle Bx_n,x_n\rangle-\left\langle\frac{4p}{a\beta_n},u\right\rangle\right)+2\lambda_n\langle w,u-x_n\rangle\\
\leq & \ \frac{a\lambda_n\beta_n}{2}\left[\sup_{u\in C}\varphi_B\left(u,\frac{4p}{a\beta_n}\right)-\sigma_C\left(\frac{4p}{a\beta_n}\right)\right]+2\lambda_n\langle w,u-x_n\rangle.
\end{align*}
\end{proof}

\begin{theorem}\label{fb-conv-single-val}
Let $(x_n)_{n \in \N}$ and $(w_n)_{n \in \N}$ be the sequences generated by Algorithm \ref{alg-fb-single-val} and $(z_n)_{n \in \N}$ the sequence defined in \eqref{average}. If $(H_{fitz})$ is fulfilled and $\limsup_{n\rightarrow+\infty}\lambda_n\beta_n<2\mu$, then the following statements are true:
\begin{itemize}
\item[(i)] for every $u\in\zer(A+D+N_C)$ the sequence $(\|x_n-u\|)_{n \in \N}$ is convergent and the series $\sum_{n \in \N} \|x_{n+1}-x_n\|^2$, $\sum_{n \in \N} \lambda_n\beta_n\langle Bx_n,x_n-u\rangle$ and $\sum_{n \in \N} \lambda_n\beta_n\|Bx_n\|^2$ are convergent as well. In particular $\lim_{n\rightarrow+\infty}\|x_{n+1}-x_n\|=0$. If, moreover, $\liminf_{n\rightarrow+\infty} \lambda_n\beta_n>0$, then $\lim_{n\rightarrow+\infty}\langle Bx_n,x_n-u\rangle=\lim_{n\rightarrow+\infty}\|Bx_n\|=0$ and every weak cluster point of $(x_n)_{n \in \N}$ lies in $C$.

\item[(ii)] $(z_n)_{n \in \N}$ converges weakly to an element in $\zer(A+D+N_C)$ as $n\rightarrow+\infty$.

\item[(iii)] if, additionally, $A$ is strongly monotone, then $(x_n)_{n \in \N}$ converges strongly  to the unique element in $\zer(A+D+N_C)$ as $n\rightarrow+\infty$.
\end{itemize}
\end{theorem}

\begin{proof} For every $u\in \zer(A+D+N_C)$, according to Lemma \ref{ineq3-fb-single-val}, there exist $a,b >0$ and $n_1 \in \N$ such that for all $n \geq n_1$ inequality \eqref{lem14} is true for $w=0$. This gives rise via Lemma \ref{Fejer-real-seq} to the statements in (i). As the sequence $(\lambda_n\beta_n)_{n \in \N}$ is bounded above, it automatically follows that $(\lambda_n\beta_n\|Bx_n\|)_{n \in \N} \in\ell^2$. Hence, (ii) and (iii) follow as consequences of Theorem \ref{fb-conv-set-v} and Theorem \ref{str-mon-fb}, respectively.
\end{proof}

\section{Tseng's Type Penalty Schemes}\label{sec3}

In this section we deal first with the monotone inclusion problem stated in Problem \ref{pr-cocoercive-single-val} by relaxing the cocoercivity of $B$ and $D$ to monotonicity and Lipschitz continuity. The iterative method we propose in this setting is a forward-backward-forward penalty scheme and relies on a method introduced by Tseng in \cite{tseng} (see \cite{br-combettes, combettes-pesquet, bauschke-book} for further details and motivations). By making use of primal-dual techniques we will be able then to employ the proposed approach when solving monotone inclusion problems involving compositions of maximally monotone operators with linear continuous ones.

\subsection{Relaxing cocoercivity to monotonicity and Lipschitz continuity}\label{fbf-sum}

We deal first we the following problem.

\begin{problem}\label{pr-Lip-single-val}
Let ${\cal H}$ be a real Hilbert space, $A:{\cal H}\rightrightarrows {\cal H}$  a maximally monotone operator, $D:{\cal H}\rightarrow{\cal H}$ a monotone and $\eta^{-1}$-Lipschitz continuous operator with $\eta > 0$, $B:{\cal H}\rightarrow{\cal H}$ a monotone and $\mu^{-1}$-Lipschitz continuous operator with $\mu>0$ and suppose that $C=\zer B \neq\emptyset$. The monotone inclusion problem to solve is
$$0\in Ax+Dx+N_C(x).$$
\end{problem}

One can notice that we have relaxed the assumptions imposed on $B$ and $D$ in Problem \ref{pr-cocoercive-single-val}, however, they are both maximally monotone, see \cite[Corollary 20.25]{bauschke-book}. It is obvious that an $\eta$-cocoercive operator with $\eta >0$ is monotone and $\eta^{-1}$-Lipschitz continuous, while the opposite is not the case. It is well-known that, due to the celebrated Baillon-Haddad Theorem (see, for instance, \cite[Corollary 8.16]{bauschke-book}), the gradient of a convex and differentiable function do not provides a counterexample in this sense, however, nonzero linear, skew and Lipschitz continuous operators are monotone, but not cocoercive. For example, when ${\cal H}$ and ${\cal G}$ are real Hilbert spaces and $L:{\cal H}\rightarrow {\cal G}$ is nonzero linear continuous, then $(x,v)\mapsto (L^*v,-Lx)$ is an example in this sense. This operator appears in a natural way when employing primal-dual approaches in the context of monotone inclusion problems as done in \cite{br-combettes} (see also \cite{b-c-h, b-h, combettes-pesquet, vu}).

\begin{algorithm}\label{alg-fbf}$ $

\noindent\begin{tabular}{rl}
\verb"Initialization": & \verb"Choose" $x_1\in{\cal H}$\\
\verb"For" $n \in \N$ \verb"set": & $p_n=J_{\lambda_n A}(x_n-\lambda_nDx_n-\lambda_n\beta_nBx_n)$\\
                                &  $x_{n+1}=\lambda_n\beta_n(Bx_n-Bp_n)+\lambda_n(Dx_n-Dp_n)+p_n$,
\end{tabular}
\end{algorithm}
where $(\lambda_n)_{n \in \N}$ and $(\beta_n)_{n \in \N}$ are sequences of positive real numbers.

\begin{remark} If $Bx=0$ for every $x\in{\cal H}$  (which corresponds to the situation $N_C(x)=\{0\}$ for all $x\in{\cal H}$), then Algorithm \ref{alg-fbf} turns out to be the error-free forward-backward-forward scheme from \cite[Theorem 2.5]{br-combettes} (see also \cite{tseng}).
\end{remark}

We start with the following technical statement.

\begin{lemma}\label{fbf-ineq1}  Let $(x_n)_{n \in \N}$ and $(p_n)_{n \in \N}$ be the sequences generated by Algorithm \ref{alg-fbf} and  let be $(u,w) \in \gr(A+D+N_C)$ such that $w=v+p+Du$, where $v\in Au$ and $p\in N_C(u)$. Then the following inequality holds for all $n\in\N$:
\begin{align}\label{lem19}
& \ \|x_{n+1}-u\|^2-\|x_n-u\|^2+\left(1-\left(\frac{\lambda_n\beta_n}{\mu} + \frac{\lambda_n}{\eta} \right)^2\right)\|x_n-p_n\|^2 \nonumber \\
\leq & \ 2\lambda_n\beta_n\left[\sup_{u\in C}\varphi_B\left(u,\frac{p}{\beta_n}\right)-\sigma_C\left(\frac{p}{\beta_n}\right)\right]+2\lambda_n\langle w,u-p_n\rangle.
\end{align}
\end{lemma}
\begin{proof} From the definition of the resolvent we have $\frac{x_n-p_n}{\lambda_n}-\beta_nBx_n-Dx_n\in Ap_n$ for every $n \in \N$ and since $v\in Au$, the monotonicity of $A$ guarantees
$$\langle x_n-p_n-\lambda_n(\beta_nBx_n+Dx_n+v),p_n-u\rangle\geq 0 \ \forall n \in \N,$$
thus
$$\langle x_n-p_n,u-p_n\rangle\leq \langle \lambda_n\beta_nBx_n+\lambda_nDx_n+\lambda_nv,u-p_n\rangle \ \forall n \in \N.$$

By using the definition of $x_{n+1}$ given in the algorithm we obtain
\begin{align*}
& \ \langle x_n-p_n,u-p_n\rangle\\
\leq & \ \langle x_{n+1}-p_n+\lambda_n\beta_nBp_n+\lambda_nDp_n+\lambda_nv,u-p_n\rangle\\
= & \ \langle x_{n+1}-p_n,u-p_n\rangle+\lambda_n\beta_n\langle Bp_n,u-p_n\rangle+\lambda_n\langle Dp_n,u-p_n\rangle+\lambda_n\langle v,u-p_n\rangle \ \forall n \in \N.
\end{align*}
From here it follows
\begin{align*}
& \ \frac{1}{2}\|u-p_n\|^2-\frac{1}{2}\|x_n-u\|^2+\frac{1}{2}\|x_n-p_n\|^2\\
\leq & \ \frac{1}{2}\|u-p_n\|^2-\frac{1}{2}\|x_{n+1}-u\|^2+\frac{1}{2}\|x_{n+1}-p_n\|^2\\
& \ +\lambda_n\beta_n\langle Bp_n,u-p_n\rangle+\lambda_n\langle Dp_n,u-p_n\rangle+\lambda_n\langle v,u-p_n\rangle \ \forall n \in \N.
\end{align*}
Since $v=w-p-Du$ and due to the fact that $D$ is monotone, we obtain for every $n \in \N$
\begin{align*}
& \ \|x_{n+1}-u\|^2-\|x_n-u\|^2\\
\leq & \ \|x_{n+1}-p_n\|^2-\|x_n-p_n\|^2+2\lambda_n\beta_n\left(\langle Bp_n,u\rangle+\left \langle p_n,\frac{p}{\beta_n} \right \rangle-\langle Bp_n,p_n\rangle
-\left \langle \frac{p}{\beta_n},u \right \rangle\right)\\
& \ +2\lambda_n\langle Dp_n-Du,u-p_n\rangle+2\lambda_n\langle w,u-p_n\rangle\\
\leq & \ \|x_{n+1}-p_n\|^2-\|x_n-p_n\|^2 +2\lambda_n\beta_n\left[\sup_{u\in C} \varphi_B\left(u,\frac{p}{\beta_n}\right)-\sigma_C\left(\frac{p}{\beta_n}\right)\right]+2\lambda_n\langle w,u-p_n\rangle.
\end{align*}
The conclusion follows, by noticing that the Lipschitz continuity of $B$ and $D$ yields
$$\|x_{n+1}-p_n\|\leq \frac{\lambda_n\beta_n}{\mu} \|x_n-p_n\|+\frac{\lambda_n}{\eta}\|x_n-p_n\| = \left (\frac{\lambda_n\beta_n}{\mu} + \frac{\lambda_n}{\eta} \right)\|x_n-p_n\| \ \forall n \in \N.$$
\end{proof}

The convergence of Algorithm \ref{alg-fbf} is stated below.

\begin{theorem}\label{fbf-conv} Let $(x_n)_{n \in \N}$ and $(p_n)_{n \in \N}$ be the sequences generated by Algorithm \ref{alg-fbf} and $(z_n)_{n \in \N}$ the sequence defined in \eqref{average}. If $(H_{fitz})$ is fulfilled and $\limsup_{n\rightarrow + \infty}\left (\frac{\lambda_n\beta_n}{\mu} + \frac{\lambda_n}{\eta} \right)<1$, then $(z_n)_{n \in \N}$ converges weakly to an element in $\zer(A+D+N_C)$ as $n\rightarrow+\infty$.
\end{theorem}

\begin{proof} The proof of the theorem relies on the following three statements:
\begin{itemize}

\item[(a)] for every $u\in\zer(A+D+N_C)$ the sequence $(\|x_n-u\|)_{n \in \N}$ is convergent;

\item[(b)] every weak cluster point of $(z_n')_{n \in \N}$, where
$$z_n'=\frac{1}{\tau_n}\sum_{k=1}^n\lambda_k p_k \ \mbox{and} \ \tau_n= \sum_{k=1}^n\lambda_k \ \forall n \in \N,$$
lies in $\zer(A+D+N_C)$;

\item[(c)] every weak cluster point of $(z_n)_{n \in \N}$ lies in $\zer(A+D+N_C)$.
\end{itemize}

In order to show (a) and (b) one has only to slightly adapt the proof of Theorem \ref{fb-conv-set-v} and this is why we omit to give further details. For (c) it is enough to prove that $\lim_{n\rightarrow+\infty} \|z_n-z_n'\|=0$ and the statement of the theorem will be a consequence of Lemma \ref{opial-passty}.

Taking $u \in \zer(A+D+N_C)$ and $w=0=v+p+Du$, where $v \in Au$ and $p \in N_C(u)$, from \eqref{lem19} we have
\begin{align*}
& \ \|x_{n+1}-u\|^2-\|x_n-u\|^2+\left(1-\left(\frac{\lambda_n\beta_n}{\mu} + \frac{\lambda_n}{\eta} \right)^2\right)\|x_n-p_n\|^2\\
\leq & \ 2\lambda_n\beta_n\left[\sup_{u\in C}\varphi_B\left(u,\frac{p}{\beta_n}\right)-\sigma_C\left(\frac{p}{\beta_n}\right)\right].
\end{align*}
As $\limsup_{n\rightarrow + \infty}\left (\frac{\lambda_n\beta_n}{\mu} + \frac{\lambda_n}{\eta} \right)<1$, we obtain by Lemma \ref{Fejer-real-seq} that $\sum_{n \in \N}\|x_n-p_n\|^2<+\infty$.

Moreover, for any $n \in \N$ it holds
\begin{align*}
\|z_n-z_n'\|^2= & \frac{1}{\tau_n^2}\left\|\sum_{k=1}^n\lambda_k(x_k-p_k)\right\|^2\leq \frac{1}{\tau_n^2}\left(\sum_{k=1}^n\lambda_k\|x_k-p_k\|\right)^2\\
\leq & \frac{1}{\tau_n^2}\left(\sum_{k=1}^n\lambda_k^2\right)\left(\sum_{k=1}^n\|x_k-p_k\|^2\right).
\end{align*}
Since $(\lambda_n)_{n \in \N} \in \ell^2 \setminus \ell^1$, taking into consideration that $\tau_n = \sum_{k=1}^n \lambda_k \rightarrow +\infty \ (n \rightarrow +\infty)$, we obtain $\|z_n-z_n'\| \rightarrow 0 \ (n \rightarrow +\infty)$.
\end{proof}

As it happens for the forward-backward penalty scheme, strong monotonicity of the operator $A$ ensures strong convergence of the sequence $(x_n)_{n \in \N}$.

\begin{theorem}\label{str-mon-fbf} Let $(x_n)_{n \in \N}$ and $(p_n)_{n \in \N}$ be the sequences generated by Algorithm \ref{alg-fbf}. If $(H_{fitz})$ is fulfilled, $\limsup_{n\rightarrow + \infty}\left (\frac{\lambda_n\beta_n}{\mu} + \frac{\lambda_n}{\eta} \right)<1$ and the operator $A$ is $\gamma$-strongly monotone with $\gamma>0$, then $(x_n)_{n \in \N}$ converges strongly to the unique element in $\zer(A+D+N_C)$ as $n\rightarrow+\infty$.
\end{theorem}

\begin{proof} Let be $u \in \zer(A+D+N_C)$ and $w=0=v+p+Du$, where $v \in Au$ and $p \in N_C(u)$. Following the lines of the proof of Lemma \ref{fbf-ineq1} one can easily show that
\begin{align*}
& \ 2\gamma\lambda_n\|p_n-u\|^2+ \|x_{n+1}-u\|^2-\|x_n-u\|^2+\left(1-\left(\frac{\lambda_n\beta_n}{\mu} + \frac{\lambda_n}{\eta} \right)^2\right)\|x_n-p_n\|^2\\
\leq & \ 2\lambda_n\beta_n\left[\sup_{u\in C}\varphi_B\left(u,\frac{p}{\beta_n}\right)-\sigma_C\left(\frac{p}{\beta_n}\right)\right] \ \forall n \in \N.
\end{align*}
The hypotheses imply the existence of $n_0 \in \N$ such that for every $n\geq n_0$
$$2\gamma\lambda_n\|p_n-u\|^2+\|x_{n+1}-u\|^2-\|x_n-u\|^2\leq 2\lambda_n\beta_n\left[\sup_{u\in C} \varphi_B \left(u,\frac{p}{\beta_n}\right) - \sigma_C\left(\frac{p}{\beta_n}\right)\right].$$
As in the proof of Theorem \ref{str-mon-fb}, from here it follows that
$$\sum_{n \in \N}\lambda_n\|p_n-u\|^2<\infty.$$

Since $(\lambda_n)_{n \in \N}$ is bounded above and $\sum_{n \in N} \|x_n-p_n\|^2<+\infty$ (see the proof of Theorem \ref{fbf-conv}), it yields
\begin{align*}
\sum_{n=1}^\infty\lambda_n\|x_n-u\|^2 \leq 2\sum_{n=1}^\infty\lambda_n\|x_n-p_n\|^2+2\sum_{n=1}^\infty\lambda_n\|p_n-u\|^2 < +\infty.
\end{align*}
As $\sum_{n \in \N} \lambda_n=+\infty$ and  $(\|x_n-u\|)_{n \in \N}$ is convergent, it follows $\lim_{n\rightarrow+\infty}\|x_n-u\|=0$.
\end{proof}

\subsection{Monotone inclusion problems involving compositions with linear continuous operators}\label{subsec3.2}

In this subsection we will show that the Tseng's type penalty scheme investigated in the previous section allows the solving of monotone inclusion problems with a more intricate formulation. The problem under consideration is the following.

\begin{problem}\label{pr-fbf-comp} Let ${\cal H}$ and ${\cal G}$ be real Hilbert spaces, $A_1:{\cal H}\rightrightarrows {\cal H}$ and $A_2:{\cal G}\rightrightarrows {\cal G}$ maximally monotone operators, $K:{\cal H}\rightarrow {\cal G}$ linear continuous operator, $D:{\cal H}\rightarrow{\cal H}$ a monotone and $\eta^{-1}$-Lipschitz continuous operator with $\eta > 0$, $B:{\cal H}\rightarrow{\cal H}$ a monotone and $\mu^{-1}$-Lipschitz continuous operator with $\mu>0$ suppose that $C=\zer B \neq\emptyset$. The monotone inclusion problem to solve is
$$0\in A_1x+K^*A_2Kx+Dx+N_C(x).$$
\end{problem}

We use the product space approach (see \cite{bauschke-book, combettes-pesquet, br-combettes, b-c-h, b-h}) in order to show that the above problem can be reformulated as the monotone inclusion problem treated in Subsection \ref{fbf-sum} in an appropriate product space. To this end we consider the real Hilbert space ${\cal H}\times{\cal G}$ endowed with the inner product
$$\langle(x,v),(x',v')\rangle_{{\cal H}\times{\cal G}}=\langle x,x'\rangle_{\cal H}+\langle v,v'\rangle_{\cal G} \ \forall (x,v),(x',v')\in {\cal H}\times{\cal G}$$
and corresponding norm. We define the following operators on ${\cal H}\times {\cal G}$. For $(x,v)\in {\cal H}\times {\cal G}$ we set
$$\widetilde A(x,v)=A_1x\times A_2^{-1}v, \ \ \widetilde D(x,v)=(Dx+K^*v,-Kx), \ \ \widetilde{B}(x,v)=(Bx,0)$$
and, for  $\widetilde{C}=C\times {\cal G} = \zer \widetilde B$,
$$N_{\widetilde{C}}(x,v) = N_C(x)\times \{0\}.$$
One can easily show that if $(x,v)\in\zer (\widetilde A + \widetilde D + N_{\widetilde{C}})$, then $x\in\zer(A_1+K^*A_2K+D+N_C)$. Conversely, when $x\in\zer(A_1+K^*A_2K+D+N_C)$, then exists $v\in A_2Kx$ such that $(x,v)\in\zer (\widetilde A + \widetilde D + N_{\widetilde{C}})$. Thus, determining the zeros of the operator $\widetilde A + \widetilde D + N_{\widetilde{C}}$ will provide a solution for the monotone inclusion problem in Problem \ref{pr-fbf-comp}.

One has that $\widetilde A$ is maximally monotone (see \cite[Proposition 20.23]{bauschke-book}), $\widetilde B$ is monotone and $\widetilde \eta$-Lipschitz continuous, where $\widetilde \eta=\sqrt{2\left(\frac{1}{\eta^2}+\|K\|^2\right)}$ and $\widetilde B$ is monotone and $\mu^{-1}$-Lipschitz continuous. All these considerations show that we are in the context of Problem \ref{pr-Lip-single-val}, thus, in order to determine the zeros of $\widetilde A + \widetilde D + N_{\widetilde{C}}$, we can use Algorithm \ref{alg-fbf}, the iterations of which read for any $n \in \N$ as follows:
$$\left\{
\begin{array}{ll}
(p_n,q_n)=J_{\lambda_n \widetilde A}\left[(x_n,v_n)-\lambda_n \widetilde D(x_n,v_n)-\lambda_n\beta_n\widetilde{B}(x_n,v_n)\right]\\
(x_{n+1},v_{n+1})=\lambda_n\beta_n\left[\widetilde{B}(x_n,v_n)-\widetilde{B}(p_n,q_n)\right]+\lambda_n\left[\widetilde D(x_n,v_n)-\widetilde D(p_n,q_n)\right]+(p_n,q_n).
\end{array}\right.$$

Since $J_{\lambda \widetilde A}(x,v)=(J_{\lambda A_1}(x),J_{\lambda A_2^{-1}}(v))$ for every $(x,v)\in{\cal H}\times{\cal G}$ and every $\lambda > 0$  (see \cite[Proposition 23.16]{bauschke-book}), this gives rise to the following iterative scheme.

\begin{algorithm}\label{alg-fbf-comp}$ $

\noindent\begin{tabular}{rl}
\verb"Initialization": & \verb"Choose" $(x_1,v_1)\in{\cal H}\times{\cal G}$\\
\verb"For" $n \in \N$ \verb"set": & $p_n=J_{\lambda_n A_1}\Big(x_n-\lambda_n(Dx_n+K^*v_n)-\lambda_n\beta_nBx_n\Big)$\\
                                 &  $q_n=J_{\lambda_nA_2^{-1}}(v_n+\lambda_nKx_n)$\\
                                 & $x_{n+1}=\lambda_n\beta_n(Bx_n-Bp_n)+\!\lambda_n(Dx_n-Dp_n) + \!\lambda_nK^*(v_n-q_n)+p_n$\\
                                 & $v_{n+1}=\lambda_nK(p_n-x_n)+q_n$,
\end{tabular}
\end{algorithm}
where $(\lambda_n)_{n \in \N}$ and $(\beta_n)_{n \in \N}$ are sequences of positive real numbers. For the convergence of this iterative scheme the following hypotheses are needed:
$$(H_{fitz}^{comp})\left\{
\begin{array}{lll}
(i) \ A_1+N_C \mbox{ is maximally monotone and }\zer(A_1+K^*A_2K+D+N_C)\neq\emptyset;\\
(ii) \ \mbox{ For each }p\in\ran (N_C), \sum_{n \in \N} \lambda_n\beta_n\left[\sup\limits_{u\in C}\varphi_B\left(u,\frac{p}{\beta_n}\right)-\sigma_C\left(\frac{p}{\beta_n}\right)\right]<+\infty;\\
(iii) \ (\lambda_n)_{n \in \N} \in\ell^2\setminus\ell^1.\end{array}\right.$$
One will see that  $(H_{fitz}^{comp})$ implies the hypotheses $(H_{fitz})$ formulated in the context of the monotone inclusion problem of finding the zeros of $\widetilde A + \widetilde D + N_{\widetilde{C}}$. Indeed, hypothesis (i) in $(H_{fitz}^{comp})$ guarantees that $\widetilde A+N_{\widetilde{C}}$ is maximally monotone and, as already seen, that $\zer (\widetilde A + \widetilde D + N_{\widetilde{C}}) \neq \emptyset$. Further, we have for all $(x,v),(x',v')\in{\cal H}\times{\cal G}$ that
$$\varphi_{\widetilde{B}}\big((x,v),(x',v')\big)=\left\{
\begin{array}{ll}
\varphi_B(x,x'), & \mbox {if } v'=0,\\
+\infty, & \mbox{otherwise}
\end{array}\right.$$
and
$$\sigma_{\widetilde{C}}(x,v)=\left\{
\begin{array}{ll}
\sigma_C(x), & \mbox {if } v=0,\\
+\infty, & \mbox{otherwise}.
\end{array}\right.$$
Moreover, $\ran N_{\widetilde{C}}=\ran N_C\times\{0\}$. Hence, condition (ii) in $(H_{fitz}^{comp})$ is nothing else than
$$\!\mbox{for each }(p,p')\in\ran (N_{\widetilde{C}}), \sum_{n \in \N} \lambda_n\beta_n\left[\sup\limits_{(u,u')\in \widetilde{C}}\varphi_{\widetilde{B}}\left((u,u'),\frac{(p,p')}{\beta_n}\right)-\sigma_{\widetilde{C}}\left(\frac{(p,p')}{\beta_n}\right)\right]<+\infty.$$
The following convergence statement is a direct consequence of Theorem \ref{fbf-conv} and Theorem \ref{str-mon-fbf}.

\begin{theorem}\label{fbf-conv-comp}
Let $(x_n)_{n \in \N}$, $(v_n)_{n \in \N}$, $(p_n)_{n \in \N}$ and $(q_n)_{n \in \N}$ be the sequences generated by Algorithm \ref{alg-fbf-comp} and $(z_n)_{n \in \N}$ the sequence defined in \eqref{average}. If $(H_{fitz}^{comp})$ is fulfilled and
$$\limsup_{n\rightarrow + \infty}\left (\frac{\lambda_n\beta_n}{\mu} + \lambda_n \sqrt{2\left(\frac{1}{\eta^2}+\|K\|^2\right)} \right)<1,$$
then $(z_n)_{n \in \N}$ converges weakly to an element in $\zer(A_1+K^*A_2K+D+N_C)$ as $n\rightarrow+\infty$. If, additionally, $A_1$ and $A_2^{-1}$ are strongly monotone, then $(x_n)_{n \in \N}$ converges strongly to the unique element in $\zer(A_1+K^*A_2K+D+N_C)$ as $n\rightarrow+\infty$.
\end{theorem}

\begin{remark}\label{lip-cocoerc} We applied the forward-backward-forward penalty scheme in the product space in order to solve monotone inclusion problems where also compositions with linear continuous operators are involved. Let us underline the fact that, even in the situation when $B$ is cocoercive and, hence, $\widetilde B$ is cocoercive, the forward-backward penalty scheme in Algorithm \ref{alg-fb-single-val} cannot be applied in this context, because the operator $\widetilde D$ is definitely not cocoercive. This is due to the presence of the skew operator $(x,v)\mapsto(K^*v,-Kx)$ in its definition. This fact provides a good motivation for formulating, along the forward-backward penalty scheme, a forward-backward-forward penalty scheme for the monotone inclusion problem investigated in this paper.
\end{remark}

\begin{remark} In the particular case $Dx=0$ for every $x\in{\cal H}$ and $Bx=0$ for every $x\in{\cal H}$ (which corresponds to the situation when $N_C(x)=\{0\}$ for every $x\in{\cal H}$) Algorithm \ref{alg-fbf-comp} turns out to be the error-free case of the forward-backward-forward scheme proposed and investigated from the point of view of its convergence in \cite[Theorem 3.1]{br-combettes}.

\end{remark}

\end{document}